\documentclass[11pt]{amsart}

\usepackage[a4paper, margin=2.5cm]{geometry}
\usepackage[expansion=false]{microtype}
\usepackage[T1]{fontenc}

\usepackage{amsmath, amssymb, amsthm, mathtools}
\usepackage{bm}

\usepackage{graphicx}
\usepackage{subcaption}
\usepackage{booktabs}
\usepackage{array}
\usepackage{multirow}
\usepackage{xcolor}
\usepackage{tikz}
\usetikzlibrary{calc, arrows.meta, positioning, fit, backgrounds}
\usepackage{listings}
\usepackage{algorithm}
\usepackage{algpseudocode}

\usepackage[numbers, sort&compress]{natbib}
\usepackage{hyperref}
\hypersetup{
    colorlinks=true,
    citecolor=blue!60!black,
    linkcolor=blue!60!black,
    urlcolor=blue!60!black
}
\usepackage{cleveref}

\theoremstyle{plain}
\newtheorem{theorem}{Theorem}[section]

\newtheorem{proposition}[theorem]{Proposition}

\theoremstyle{definition}
\newtheorem{definition}[theorem]{Definition}
\newtheorem{remark}[theorem]{Remark}

\newcommand{\R}{\mathbb{R}}
\newcommand{\bb}{\bm{b}}

\newcommand{\bx}{\bm{x}}

\newcommand{\bz}{\bm{z}}

\newcommand{\bphi}{\bm{\varphi}}
\newcommand{\dom}{\Omega}
\newcommand{\bdry}{\partial\Omega}
\newcommand{\mtil}{\tilde{m}}
\newcommand{\eps}{\varepsilon}

\newcommand{\norm}[1]{\left\lVert#1\right\rVert}
\newcommand{\snorm}[1]{\left\lvert#1\right\rvert}

\newcommand{\Lp}[1]{L^{#1}}
\newcommand{\Hone}{H^{1}}
\newcommand{\Hhalf}{H^{1/2}}
\newcommand{\Hminusone}{H^{-1}}
\newcommand{\HhalfBdry}{\Hhalf(\bdry)}

\newcommand{\Lpinn}{L_{\mathrm{pinn}}}        
\newcommand{\Lcons}{L^{\ast}_{\mathrm{c}}}    
\newcommand{\Lgam}{L^{\ast}_{\gamma}}         
\newcommand{\enorm}[1]{\vert\!\vert\!\vert #1 \vert\!\vert\!\vert}
\newcommand{\Rop}{\mathcal{R}}
\newcommand{\Lop}{\mathcal{L}}

\definecolor{cout}{RGB}{120,90,160}

\title[Consistent CutPINNs for Convection-Diffusion]{Consistent CutPINNs for Convection-Diffusion Equations on Curved Level-Set Domains}

\author{Maneesh Kumar Singh}
\address{Department of Mathematics, SRM Institute of Science and Technology,
Kattankulathur, India}
\email{maneeshs@srmist.edu.in}

\date{\today}



\begin{document}

\begin{abstract}
We present an a priori error analysis of consistent-loss PINNs for
stationary convection-diffusion equations on curved level-set domains. The
standard mean-squared interior loss fails in the convection-dominated
regime: the solution develops an $O(\eps)$ boundary layer in which the
pointwise residual grows like $\eps^{-1}$, so the loss is dominated by the
few collocation points inside the layer and leaves the smooth bulk
unresolved. We remove this mismatch by penalising the interior residual in a
discrete $\Lp{\gamma}$ norm with $\gamma = 1 + 1/\log\mtil$, a computable
surrogate for the $\Hminusone$ stability term, and imposing the boundary
condition through a discrete $\HhalfBdry$ trace norm, which treats flat and
curved geometries uniformly. Under Besov regularity assumptions we prove a
single a priori $\Hone$ error bound, valid for all interior exponents
$\gamma \in (1,2]$, with an optimal recovery rate governed by a cut-cell
floor $1/(2\gamma)$ specific to the curved geometry. Numerical experiments
on a rectangle and a disk at $\eps = 2^{-s}$, $s \in \{2,4,6\}$, confirm the
analysis: as the layer sharpens, the $\Lp{2}$ interior loss becomes
seed-fragile while the $\Lp{\gamma}$ interior trains reliably, the interior
norm being the decisive factor in convergence.
\end{abstract}

\maketitle

\section{Introduction}
\label{sec:intro}

\subsection{Background and motivation}

The stationary convection-diffusion equation
\begin{equation}
\label{eq:cdr-intro}
    -\eps \Delta u + \bb \cdot \nabla u + c\,u = f \quad \text{in } \dom,
    \qquad u = g \quad \text{on } \bdry,
\end{equation}
on a bounded domain $\dom \subset \R^{2}$ underlies the modelling of many
transport phenomena, from the spread of a pollutant and charge transport in
semiconductors to the Oseen system obtained by linearising incompressible
flow. How the solution behaves is set by the ratio of the diffusion
coefficient $\eps$ to the convection magnitude $\norm{\bb}_{\infty}$. When
that ratio becomes small, the regime we call convection-dominated, $u$
develops thin layers near the boundary or in the interior, and approximating
it accurately has long been difficult for classical
discretisations~\cite{roos2008robust,singh2020numerical,singh2021unified} and for neural-network
solvers~\cite{krishnapriyan2021characterizing,frerichs2026loss} alike.

Among neural-network approaches, physics-informed neural networks
(PINNs)~\cite{raissi2019} hold particular appeal here: because the governing
equation is enforced at scattered points rather than over a triangulation,
no mesh is required, which spares the cost of building and repeatedly
refining a conforming grid on a geometrically intricate domain. The standard
formulation nonetheless meets two separate obstacles in the present setting.
The first is a question of norms. On a general bounded domain the
$\Lp{2}(\bdry)$ boundary penalty fails to control the $\HhalfBdry$ trace
norm that governs well-posedness in $\Hone(\dom)$, so a loss built on it is
not consistent with the functional-analytic structure of the problem. The
second is specific to convection dominance. Writing the pointwise residual
as $\Rop(v)(\bx) := -\eps \Delta v + \bb \cdot \nabla v + c v - f$, its
values reach order $\eps^{-1}$ throughout the $O(\eps)$-wide layer, so the
few collocation points of $X$ that land there dominate the squared-residual
sum. Minimising such a loss drives the network to resolve the layer while
neglecting the smooth bulk, which is exactly where the $\Hone$ error is
concentrated.

The idea of training a network to satisfy a differential equation predates
the modern formulation~\cite{lagaris1998artificial}, but it is the
collocation-based variant of~\cite{raissi2019} that has driven the recent
activity, surveyed in~\cite{cuomo2022,plankovskyy2025review}. A substantial
theory has since accumulated. Convergence for linear second-order elliptic
and parabolic equations was established in~\cite{shin2020convergence},
generalisation and approximation error were quantified
in~\cite{mishra2023estimates,de2022generic}, and a unified error-analysis
framework was given in~\cite{zeinhofer2025unified}; the numerical-analysis
perspective is reviewed in~\cite{de2024numerical}. This analysis draws on
approximation theory for neural networks~\cite{devore2021neural} and on
optimal recovery and nonlinear approximation in Besov
classes~\cite{cohen2022optimal,devore1993besov,novak2006function}, the same
machinery that underlies the consistent-loss approach we adopt below.
Methodological variants are numerous: variational and Deep Ritz
formulations~\cite{duan2022convergence}, domain-decomposition and extended
PINNs~\cite{hu2022extended}, and adaptive collocation
strategies~\cite{visser2026pacmann}. Even so, controlled comparisons with
finite element solvers~\cite{grossmann2023pinn} show that a standard PINN is
frequently less accurate per degree of freedom, most visibly when the
solution carries boundary layers or sharp gradients, which is the regime of
concern here.

\subsection{Prior work and the open intersection}
Now, we discussed the previous work related to Consistent PINN and the motivation behind this research. \\

\textbf{Consistent PINN losses on flat domains:}
The norm mismatch on the boundary was first resolved by Bonito, DeVore,
Petrova, and Siegel~\cite{bonito2025} for the unit cube. In place of the
$\Lp{2}(\bdry)$ penalty they introduced a discrete Gagliardo double sum over
the boundary collocation points, and they showed this discrete functional to
be equivalent to the continuous $\HhalfBdry$ norm over suitable Besov model
classes. A loss assembled from it inherits provable $\Hone$ recovery rates.
Their treatment was confined to elliptic problems on flat geometry, and the
interior residual was always measured in $\Lp{2}$.

\textbf{Consistent losses on curved cut domains:}
Carrying the consistent framework onto curved geometry, Singh~\cite{singh2026cutpinn}
treated bounded domains given as the negative level set of a $C^{2}$
function. The technical heart of that work is a chord-arc lemma establishing
that the discrete Gagliardo seminorm, evaluated on $m$ points spaced equally
in arc length along a $C^{2}$ Jordan curve, stays equivalent to the
continuous $\HhalfBdry$ trace norm; the resulting $\Hone$ rates match
those of~\cite{bonito2025} up to constants set by the curvature and the
chord-arc parameter. Two interior penalties were in fact defined there, one
using $\Lp{\gamma}$ with $\gamma = 1 + 1/\log\mtil$ and one using $\Lp{2}$,
but on the smooth elliptic benchmarks of that paper the two choices returned
all but identical $\Hone$ error, and the theory and the experiments of that paper
were both developed for the $\Lp{2}$ penalty. What was an immaterial
modelling choice in the elliptic case is, as we show below, the deciding
factor once convection dominates.

\textbf{PINNs for convection-dominated problems:}
The instability of standard PINNs on convection-dominated PDEs was
documented by Krishnapriyan et al.~\cite{krishnapriyan2021characterizing},
who linked it to pathologies of the optimisation landscape near the layer.
The responses in the literature pull in different directions: coupling a
PINN to an SUPG-stabilised finite element reference~\cite{cengizci2026pinn};
neural prediction of stabilisation parameters for convection-dominated
problems~\cite{yadav2024artificial,yadav2024convstabnet};
$hp$-variational PINNs in one dimension or on
rectangles~\cite{kumar2025variational,anandh2025improving,anandh2025fastvpinns};
direct PINN treatment of boundary-layer problems~\cite{raina2026application};
reweighting the loss and adapting the sample
distribution~\cite{frerichs2026loss}; and unfitted finite-element neural
schemes carrying a CutFEM-type test space~\cite{li2026}, built on aggregated
unfitted finite element technology~\cite{badia2018aggregated}. Individually
these target convection, geometry, or loss design, yet none unites a
consistent-loss formulation with a curved cut domain and a
convection-dominated regime while remaining a pure PINN.

\textbf{Consistent PINN for convective operators:}
The consistent-loss programme has by now been carried to higher-order
elliptic operators~\cite{mishra2026consistent} and to obstacle problems in
mixed form~\cite{khan2026mixed}, but its only treatment of a convective
operator to date is the consistent-PINN study of the stationary
Oseen system by Mishra and Khan~\cite{mishra2026structure}, which derives
its loss from the stability structure of a convective problem and proves
optimal recovery in $\Hone$ for the velocity together with $\Lp{2}$ for the
pressure. We share its premise, a loss read off the continuous stability
estimate of a problem carrying a first-order term, and it confirms that the
optimal-recovery apparatus of~\cite{bonito2025} survives the passage to such
operators. Two features separate it from the layer-dominated setting we
study. Its analysis and numerics sit in the balanced regime, with unit
viscosity $\nu = 1$ on the unit square, giving a P\'eclet number of order
one, smooth layer-free solutions, and an $\Lp{2}$ interior penalty
throughout; the overweighting mechanism we are concerned with simply does
not occur there, and nothing forces a departure from $\Lp{2}$. Its geometry
is moreover a flat box, so cut cells and curved-boundary trace equivalence
never enter. Pushing the consistent loss into convection dominance forces
the interior norm to change, and pushing it onto curved cut domains brings
in the trace machinery of~\cite{singh2026cutpinn}; the $\Lp{\gamma}$ handling of a
convective residual that~\cite{mishra2026structure} carries out at the
$\Lp{2}$ level is, here, precisely the part that must be reworked rather
than reused.

\textbf{The open intersection.}
Two bodies of work thus stop short of meeting. On the consistent-loss side,
\cite{bonito2025} covers only Poisson on flat geometry and
\cite{singh2026cutpinn} only Poisson on cut domains; on the
convection-diffusion PINN side,
\cite{cengizci2026pinn, kumar2025variational, anandh2025improving, frerichs2026loss}
remain on flat or rectangular domains and outside the consistent-loss
paradigm, relying instead on FEM coupling, Petrov--Galerkin variational
forms, or hand-tuned reweighting. A consistent-PINN account of
convection-diffusion appears not to have been given even on a rectangle, and
the curved cut-domain case is open as well. The present paper supplies both
at once.

\subsection{Contributions}
The main contribution of our work is listed below. 

\begin{enumerate}
    \item \textbf{First consistent-loss treatment of convection-diffusion.}
          We are not aware of any prior study that places the stationary
          convection-diffusion equation inside the consistent-PINN
          framework, on either flat or curved geometry. Existing
          consistent-loss work treats Poisson-type problems
          only~\cite{bonito2025,singh2026cutpinn}, while the convection-diffusion
          PINN literature~\cite{cengizci2026pinn,frerichs2026loss,kumar2025variational,
          anandh2025improving} operates outside the consistent-loss paradigm, using
          SUPG-stabilised hybrids, Petrov--Galerkin variational PINNs, or
          empirical loss reweighting. The present work occupies this
          intersection.

    \item \textbf{A norm comparison identifying the interior norm as the
          governing choice.} We compare three losses on a common ansatz:
          standard PINN ($\Lp{2}$ interior, $\Lp{2}$ boundary); the
          consistent loss ($\Lp{2}$ interior, discrete $\HhalfBdry$
          boundary); and the $\Lp{\gamma}$ loss ($\Lp{\gamma}$ interior with
          $\gamma = 1 + 1/\log\mtil$, discrete $\HhalfBdry$ boundary),
          across a rectangle and a disk at $\eps = 2^{-s}$,
          $s \in \{2,4,6\}$. The comparison shows that, in the
          convection-dominated regime, the consistent \emph{boundary} norm
          alone is not enough, and that the correction that matters is in
          the \emph{interior} norm. The exponent $\gamma = 1 + 1/\log\mtil$
          is the $\Hminusone$-side recovery exponent of~\cite{bonito2025},
          which we find transfers to the convection-dominated setting;
          since it is fixed a priori by the interior budget $\mtil$, it
          leaves the method with no interior-norm tuning parameter. On seeds
          that converge, the three losses are comparable in $\Hone$
          accuracy, and the separation between them is in \emph{reliability}.
          As the layer sharpens, the $\Lp{2}$-interior losses become
          seed-fragile and diverge on a growing fraction of initialisations,
          whereas $\Lp{\gamma}$ remains dependable across every seed on both
          domains. The interior norm governs whether convection-dominated
          training converges at all, not only the accuracy of a converged
          run.

    \item \textbf{An a priori $\Hone$ error bound.} We prove an a priori
          bound for the consistent CutPINN approximation
          of~\eqref{eq:cdr-intro} of the form
          \begin{equation*}
              \norm{u - v}_{\Hone(\dom)}
              \;\leq\; C(\eps^{-1}, \kappa_{\max}, \norm{\bb}_{\infty}, c)\,(\log\mtil)
              \Bigl[ \Lgam(v)^{1/2}
              + (1 + \norm{v}_{U}) R(\mtil, m) \Bigr],
          \end{equation*}
          where $R(\mtil, m) = \mtil^{-\alpha_{\gamma}} + m^{-\bar{s}+1}$
          and $\alpha_{\gamma}$ is the recovery rate of the $\Lp{\gamma}$
          interior discretisation on $\dom$. This extends the consistent-PINN
          analysis to convection-diffusion, carrying the standard
          $\eps^{-1}$ stability constant and the $\Lp{\gamma}$ interior
          step; it specialises to the rectangle when $\kappa_{\max} = 0$ and
          recovers~\cite[Theorem 7]{singh2026cutpinn} in the elliptic limit
          $\bb = 0$, $\eps = 1$, $\gamma \to 2$.
\end{enumerate}

\subsection{Outline}

\Cref{sec:formulation} states the problem and the discrete trace norm used
throughout. \Cref{sec:method} presents the three
loss variants ($\Lpinn$, $\Lcons$, $\Lgam$) and the practical training
schedule. \Cref{sec:analysis} develops the a priori bound on flat and
curved domains, drawing on results of~\cite{bonito2025,singh2026cutpinn}
as black boxes. \Cref{sec:numerics} reports three numerical experiments:
convergence in $\eps$ on the rectangle (\Cref{sec:exp-rectangle}) and on the
disk (\Cref{sec:exp-disk}); and spatial error distributions on both domains
(\Cref{sec:exp-spatial}). \Cref{sec:conclusion} closes with future
directions, including the parabolic extension.

\section{Problem Formulation}
\label{sec:formulation}

In this section, we  establish the continuous problem and the discrete approximation, on which  the
training pipeline is built from. Section~\ref{sec:cdr-pde} states the
convection-diffusion boundary-value problem on a curved cut domain and the
coefficient assumptions under which it is well posed, and records the
precise sense in which we work in the convection-dominated regime.
Section~\ref{sec:collocation} introduces the interior and boundary
collocation sets that replace the continuous data, and
Section~\ref{sec:boundary-disc} defines the discrete trace norm through
which the boundary condition is later enforced. Together these supply the
ingredients used throughout the loss construction of
Section~\ref{sec:method} and the analysis of Section~\ref{sec:analysis}.

\subsection{The convection-diffusion equation on a cut domain}
\label{sec:cdr-pde}

We begin with the continuous problem. Let $\dom \subset \R^{2}$ be a bounded
domain defined implicitly by a
$C^{2}$ level-set function $\bphi : \R^{2} \to \R$,
\begin{equation}
    \dom = \{\bx \in \R^{2} : \bphi(\bx) < 0\},
    \qquad
    \bdry = \{\bx \in \R^{2} : \bphi(\bx) = 0\},
\end{equation}
with $|\nabla\bphi| > 0$ on $\bdry$. The boundary is then a $C^{2}$ simple
closed curve with finite length $L$ and bounded maximum curvature
$\kappa_{\max}$. Consider the stationary convection-diffusion problem
\begin{equation}
\label{eq:cdr}
\left\{
\begin{aligned}
    -\eps \Delta u + \bb \cdot \nabla u + c\,u &= f \qquad \text{in } \dom, \\
                                            u &= g \qquad \text{on } \bdry,
\end{aligned}
\right.
\end{equation}
with $\eps > 0$, convection field $\bb \in C^{1}(\bar{\dom}; \R^{2})$,
reaction coefficient $c \in \Lp{\infty}(\dom)$ satisfying $c \geq 0$ and
$c - \tfrac{1}{2}\nabla\cdot\bb \geq \kappa_{0} > 0$, source $f \in
\Lp{2}(\dom)$, and Dirichlet data $g \in \HhalfBdry$. With these
hypotheses in force, \eqref{eq:cdr} is well posed and possesses a unique
weak solution $u \in \Hone(\dom)$.

Throughout, $C$ denotes a generic positive constant whose value may change
from line to line. It may depend on the domain $\dom$, its boundary length
$L$ and maximum curvature $\kappa_{\max}$, the coefficients of
\eqref{eq:cdr} through $\norm{\bb}_{\infty}$, $\norm{c}_{\infty}$ and the
coercivity parameter $\kappa_{0}$, the smoothness and integrability indices
$\bar{s}, \gamma$, the polynomial degree $r$, and the Stein extension
constant, but never on the collocation budgets $\mtil, m$, the exact
solution $u$, or the trial function $v$. We write $a \lesssim b$ for
$a \leq C b$ with such a constant, and $a \asymp b$ when $a \lesssim b$ and
$b \lesssim a$. The dependence on $\eps$ is tracked explicitly and not
absorbed into $C$, since the $\eps^{-1}$ factor in the stability constant
is central to the convection-dominated analysis.

\begin{remark}
The regime of interest in this work is convection-dominated but not
singularly perturbed in the asymptotic sense. Here $\eps$ ranges over
$2^{-s}$ for $s \in \{2, 4, 6\}$, giving a global P\'eclet number
$|\bb| L / \eps$ of order $\sim 4$ to $\sim 60$, which is solidly
convection-dominated. At the budget used in the experiments the mesh P\'eclet
$|\bb| h / \eps$ with $h \sim \mtil^{-1/2}$ stays $O(1)$, so the
$O(\eps)$-wide layer is \emph{marginally resolved} by the collocation set,
with on the order of $\mtil\eps$ points falling in it. In this regime the
failure of the $\Lp{2}$ interior loss is one of \emph{overweighting} a
layer that is resolved, not of failing to see an unresolved one, and it is
this overweighting that the $\Lp{\gamma}$ correction addresses; see
Remark~\ref{rem:l2-fails}. The constant in the main a priori bound scales
as $\eps^{-1}$, which is exactly the $\Hone$ stability constant classical
convection-diffusion theory assigns to this regime~\cite{roos2008robust},
so the bound is the natural one for the norm and the regime studied here;
$\eps$-uniformity in the singular-perturbation sense concerns a different
asymptotic regime and is taken up as an extension in 
\Cref{sec:conclusion}.
\end{remark}

\subsection{Collocation data}
\label{sec:collocation}

In the PINN formulation the data $f$ and $g$ are accessed only at finite
collocation sets,
\begin{equation*}
    X = \{\bx_{1}, \ldots, \bx_{\mtil}\} \subset \dom,
    \qquad
    Z = \{\bz_{1}, \ldots, \bz_{m}\} \subset \bdry,
\end{equation*}
with $m = \lfloor \sqrt{\mtil} \rfloor$, the $d = 2$ scaling
of~\cite{bonito2025}. The interior set $X$ is drawn by rejection: candidate
points are taken uniformly over the bounding box of $\dom$ and kept when
$\bphi(\bx) < 0$. The boundary set $Z$ is laid down at equal arc-length
spacing along $\bdry$ under a parametrisation $\chi : [0, L] \to \bdry$.

\subsection{Discrete boundary norm}
\label{sec:boundary-disc}
Write $r_{j} := (g - v)(\bz_{j})$ for the boundary residual sampled at the
collocation site $\bz_{j} \in Z$. From these $m$ values we form a discrete
counterpart of the $\HhalfBdry$ trace norm in two pieces, an $\Lp{2}$ part
and a Gagliardo seminorm part:
\begin{align}
    \norm{g - v}^{\ast 2}_{\Lp{2}(\bdry)}
        &= \frac{L}{m} \sum_{j=1}^{m} r_{j}^{2},
    \\
    \snorm{g - v}^{\ast 2}_{\HhalfBdry}
        &= \frac{1}{m^{2}} \sum_{\substack{i,j=1 \\ i \neq j}}^{m}
              \frac{(r_{i} - r_{j})^{2}}{|\bz_{i} - \bz_{j}|^{2}},
\end{align}
their sum
$\norm{g - v}^{\ast 2}_{\HhalfBdry} = \norm{g - v}^{\ast 2}_{\Lp{2}(\bdry)}
+ \snorm{g - v}^{\ast 2}_{\HhalfBdry}$
being the full discrete trace norm. When the $\bz_{j}$ are equispaced in
arc length along a $C^{2}$ Jordan curve, this functional stays norm-equivalent
to the continuous $\HhalfBdry$ norm over the relevant Besov model classes;
Theorem~\ref{thm:boundary-norm-equiv} states the equivalence we rely on.

\section{The Consistent CutPINN Loss Hierarchy}
\label{sec:method}

We compare three loss functionals on a common neural network ansatz
$v = v_{\theta}$. All three share the discrete pointwise PDE residual
\begin{equation}
\label{eq:residual}
    \Rop(v; \bx) := -\eps \Delta v(\bx) + \bb(\bx) \cdot \nabla v(\bx)
                  + c(\bx) v(\bx) - f(\bx).
\end{equation}

\subsection{Three loss functionals}

The three losses below differ only in the norms used to penalise the
interior and boundary residuals, and form a hierarchy: each replaces one
$\Lp{2}$ penalty of its predecessor by the norm in which the corresponding
residual is actually controlled.

\begin{definition}[Loss hierarchy]
\label{def:losses}
For a test function $v \in C^{2}(\bar{\dom})$, define
\begin{align}
\Lpinn(v) &= \frac{1}{\mtil} \sum_{i=1}^{\mtil} \Rop(v; \bx_{i})^{2}
           + \frac{1}{m} \sum_{j=1}^{m} (g - v)(\bz_{j})^{2},
\label{eq:Lpinn} \\
\Lcons(v) &= \frac{1}{\mtil} \sum_{i=1}^{\mtil} \Rop(v; \bx_{i})^{2}
           + \norm{g - v}^{\ast 2}_{\HhalfBdry},
\label{eq:Lcons} \\
\Lgam(v) &= \Biggl( \frac{1}{\mtil} \sum_{i=1}^{\mtil}
                 \snorm{\Rop(v; \bx_{i})}^{\gamma} \Biggr)^{2/\gamma}
           + \norm{g - v}^{\ast 2}_{\HhalfBdry},
\label{eq:Lgamma}
\end{align}
where $\gamma := 1 + 1/\log\mtil$.
\end{definition}

The three losses correspond to three points on a methodological hierarchy.
$\Lpinn$ is the standard PINN loss adapted to~\eqref{eq:cdr}: $\Lp{2}$
interior, $\Lp{2}$ boundary. $\Lcons$ replaces the boundary term by the
discrete $\HhalfBdry$ norm of \Cref{sec:boundary-disc}, applied to the
convection-diffusion residual. $\Lgam$ further replaces the interior term
by a discrete $\Lp{\gamma}$ norm with $\gamma$ a logarithmic perturbation of
$\Lp{1}$ governed by the interior budget $\mtil$.

\begin{remark}[Why $\gamma = 1 + 1/\log\mtil$]
\label{rem:gamma-choice}
The choice of $\gamma$ is dictated by the optimal recovery theory of
\cite{bonito2025}, where it arises as the boundary between two regimes in
the recovery exponent on Besov model classes. Concretely, the $\Lp{\gamma}$
norm interpolates between $\Lp{1}$ and $\Lp{2}$ in such a way that the
discrete $\Lp{\gamma}$ functional remains equivalent to the continuous
$\Lp{\gamma}$ norm on the model class
$U(B^{s}_{\gamma}(\Lp{\gamma}(\dom)))$, with $s$ matching the smoothness of
the source~$f$, while the $\Lp{\gamma}$ norm in turn controls the
$\Hminusone$ residual through the Sobolev embedding used in the analysis
below. In our numerical setting $\gamma$ is only slightly above $1$:
at $\mtil = 1600$ we have $\gamma = 1 + 1/\log(1600) \approx 1.136$. The
practical effect is to dampen the layer overweighting of $\Lp{2}$ without
collapsing entirely to $\Lp{1}$, which would be too insensitive to the
bulk.
\end{remark}

\begin{remark}[Why $\Lp{2}$ fails on convection-dominated problems]
\label{rem:l2-fails}
The pointwise residual at a layer point $\bx_{\mathrm{layer}}$ scales as
$\Rop(v) \sim \eps^{-1}$. The $\Lp{2}$ functional therefore amplifies a
single layer residual by $\eps^{-2}$ relative to a bulk residual of order
$1$. With $\mtil$ uniform interior points and an $O(\eps)$-wide layer, the
expected number of layer points in $X$ is $O(\mtil \eps)$, and the
$\Lp{2}$ loss is dominated by the contribution
$\eps^{-2} \cdot \mtil \eps / \mtil = \eps^{-1}$ from these few points.
The optimisation therefore fits the layer at the expense of the smooth
bulk. The $\Lp{\gamma}$ functional with $\gamma$ slightly above $1$
amplifies the layer residual by $\eps^{-\gamma}$, a factor of $\eps^{2 -
\gamma}$ smaller; the bulk regains its appropriate weight in the loss.
\end{remark}

\subsection{Network architecture and training}
\label{sec:network}

We use a fully-connected MLP $v_{\theta} : \R^{2} \to \R$ with five hidden
layers of width $50$, $\tanh$ activations, and Xavier initialisation,
identical to the architecture used in~\cite{singh2026cutpinn}. No Fourier
feature embedding~\cite{tancik2020fourier} and no layer-aware ansatz are used in
the proposed framework.

Training proceeds in two phases. First, the AdamW optimiser is run for
$2{,}000$ iterations with learning rate $8 \times 10^{-5}$, gradient
clipping at norm $0.8$, and ReduceLROnPlateau scheduling. Second, L-BFGS is
run for up to $2{,}000$ iterations as a final polish. We observe that pure
L-BFGS, the training schedule used in~\cite{singh2026cutpinn}, fails on an appreciable fraction
of seeds at $s \geq 6$. The AdamW
phase first navigates the flat directions
of the $\Lp{\gamma}/\HhalfBdry$ loss landscape, after which L-BFGS
converges quickly.

\begin{algorithm}[t]
\caption{Consistent CutPINN training for \eqref{eq:cdr}}
\label{alg:train}
\begin{algorithmic}[1]
\Require level-set $\bphi$, data $f, g$, budgets $\mtil, m$, interior
exponent $\gamma = 1 + 1/\log\mtil$
\State sample interior points $X$ by rejection on the bounding box of
$\dom$ ($\bphi < 0$); place $m$ arc-length-equispaced boundary points $Z$
\State initialise network $v_{\theta}$ (5 layers, width 50, $\tanh$, Xavier)
\State assemble the loss $\Lgam(v_{\theta})$ of Definition~\ref{def:losses}: discrete
$\Lp{\gamma}$ interior residual and discrete $\HhalfBdry$ boundary term
\For{$2000$ iterations} \Comment{phase 1: exploration}
    \State AdamW step on $\Lgam$ (lr $8\times10^{-5}$, grad-clip $0.8$,
    ReduceLROnPlateau)
\EndFor
\For{up to $2000$ iterations} \Comment{phase 2: refinement}
    \State L-BFGS step on $\Lgam$
\EndFor
\State \Return trained network $v_{\theta}$
\end{algorithmic}
\end{algorithm}

The pipeline is summarised schematically in \Cref{fig:pipeline}.

\begin{figure}[!htb]
\centering
\definecolor{cdom}{RGB}{70,110,160}
\definecolor{closs}{RGB}{40,140,120}
\definecolor{copt}{RGB}{175,110,60}
\resizebox{0.79\textwidth}{!}{%
\begin{tikzpicture}[
    >={Stealth[length=2.4mm]},
    font=\small,
    card/.style={rounded corners=3pt, draw=#1, line width=0.9pt, fill=#1!6,
                 inner sep=5pt, align=center, text width=32mm, minimum height=9mm},
    sumnode/.style={circle, draw=closs, line width=1pt, fill=closs!10,
                    inner sep=2pt, minimum size=10mm},
    lbl/.style={font=\footnotesize\itshape, text=#1},
    flow/.style={->, line width=1pt, draw=black!40, shorten >=1pt, shorten <=1pt},
    gen/.style={->, line width=1pt, draw=cdom!55, shorten >=1pt, shorten <=1pt},
]

\node[card=cdom] (dom) {Level-set domain\\[1pt] $\dom=\{\bphi<0\}$};
\node[card=cdom, below left=8mm and 1mm of dom] (X) {Interior points $X=\{\bx_i\}_{i=1}^{\mtil}$};
\node[card=cdom, below right=8mm and 1mm of dom] (Z) {Boundary points $Z=\{\bz_j\}_{j=1}^{m}$};

\node[card=closs, below=9mm of X] (int)
     {Interior $\Lp{\gamma}$:\ $\bigl(\tfrac1\mtil\!\sum_i|\Rop(\bx_i)|^{\gamma}\bigr)^{2/\gamma}$};
\node[card=closs, below=9mm of Z] (bdy)
     {Boundary: $\norm{g-v}^{\ast2}_{\HhalfBdry}$};
\node[sumnode] (sum) at ($(int)!0.5!(bdy) + (0,-10mm)$) {$\Lgam$};

\node[card=copt, below=9mm of sum] (adam)
     {Phase 1 \ AdamW \scriptsize(explore)};
\node[card=copt, below=5mm of adam] (lbfgs)
     {Phase 2 \ L\textendash BFGS \scriptsize(refine)};
\node[card=cout, below=8mm of lbfgs] (out) {Trained $v_\theta$};

\draw[gen] (dom) to[out=-150,in=90] (X);
\draw[gen] (dom) to[out=-30,in=90] (Z);
\draw[flow] (X) -- (int);
\draw[flow] (Z) -- (bdy);
\draw[flow] (int) to[out=-90,in=150] (sum);
\draw[flow] (bdy) to[out=-90,in=30] (sum);
\draw[flow] (sum) -- (adam);
\draw[flow] (adam) -- (lbfgs);
\draw[flow] (lbfgs) -- (out);

\node[lbl=cdom!75, left=3mm of dom] {setup};
\node[lbl=closs!80, left=3mm of int] {loss};
\node[lbl=copt!85, left=3mm of adam] {training};

\end{tikzpicture}
}
\caption{The consistent CutPINN pipeline of Algorithm~\ref{alg:train}. The
level-set geometry fixes the interior and boundary collocation sets; the
loss $\Lgam$ pairs a discrete $\Lp{\gamma}$ interior residual with a
discrete $\HhalfBdry$ boundary term; and training runs an AdamW exploration
phase followed by an L-BFGS refinement phase. The interior exponent
$\gamma = 1 + 1/\log\mtil$ is the only departure from a standard PINN loss.}
\label{fig:pipeline}
\end{figure}
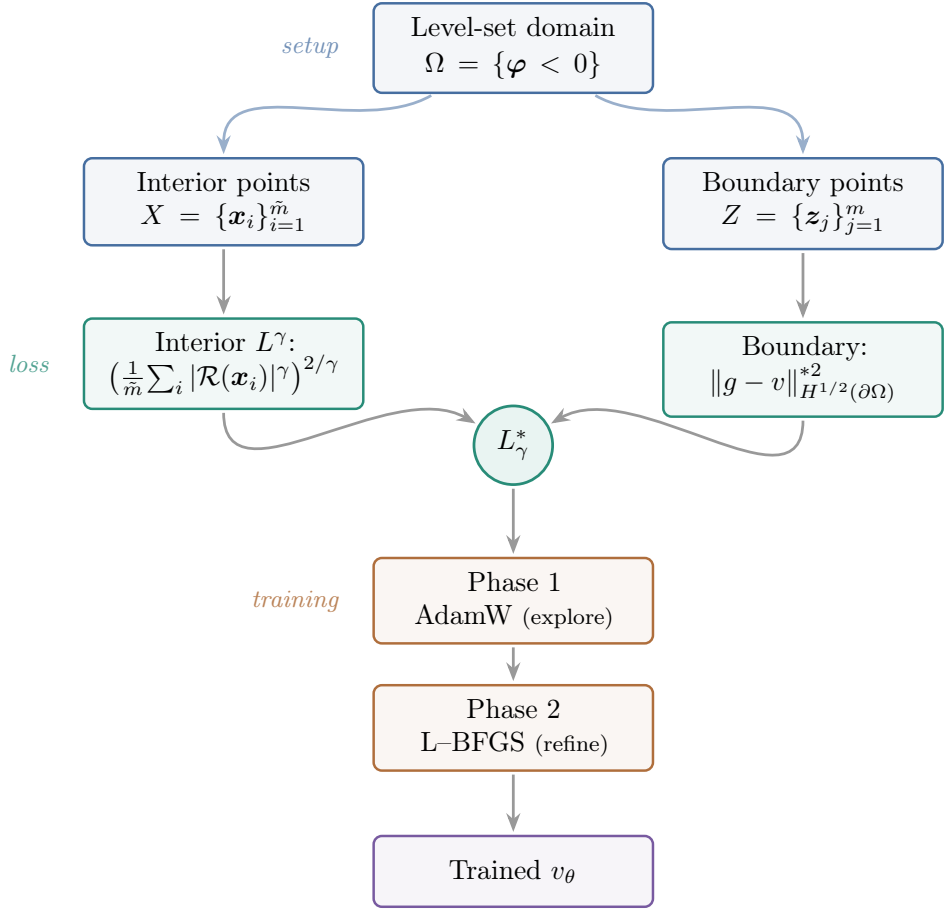

\section{Error Analysis}
\label{sec:analysis}

The goal of this section is the a priori $\Hone$ bound of
Theorem~\ref{thm:apriori-lgamma}, controlling the error of any trial network
by its computable loss together with discretisation terms that vanish as the
collocation budgets grow. The argument proceeds in four steps, and the
section is organised around them. Section~\ref{sec:cont-stability} records
the continuous stability estimate that reduces the $\Hone$ error to an
$\Hminusone$ interior residual and an $\HhalfBdry$ boundary residual.
Section~\ref{sec:boundary-equiv} recalls the discrete trace-norm equivalence
that handles the boundary residual, which we take from
\cite{singh2026cutpinn} as a black box. Section~\ref{sec:interior-comp}
proves the new ingredient, a discrete $\Lp{\gamma}$ comparison on the cut
domain, where the curved geometry enters through a cut-cell floor on the
recovery rate. Section~\ref{sec:apriori} assembles these into the main
bound.

\subsection{Continuous stability}
\label{sec:cont-stability}

We first isolate the stability of the continuous problem, which converts the
$\Hone$ error into residuals the discrete losses can measure.
The continuous problem~\eqref{eq:cdr} admits an a priori bound in the
$\Hone$ norm of the form
\begin{equation}
\label{eq:continuous-stability}
    \norm{u - v}_{\Hone(\dom)}
    \;\leq\; C_{\mathrm{stab}}(\eps^{-1}) \,
    \Bigl[ \norm{f - \Lop v}_{\Hminusone(\dom)}
         + \norm{g - v}_{\HhalfBdry} \Bigr]
\end{equation}
for every $v \in \Hone(\dom)$, where $\Lop v := -\eps \Delta v + \bb \cdot
\nabla v + c v$ and the stability constant $C_{\mathrm{stab}}$ depends on
$\eps^{-1}$, the bounded coefficients, and the coercivity parameter
$\kappa_{0}$ of \Cref{sec:cdr-pde}. The factor $\eps^{-1}$ arises because
coercivity controls the $\eps$-weighted energy norm, and recovering the full
$\Hone$ norm from it costs a factor $\eps^{-1/2}$ in each of two places: in
the duality pairing of the residual against the error, and in passing from
the energy norm of the error to its $\Hone$ norm. The resulting $\eps^{-1}$
is the standard convection-diffusion $\Hone$ stability constant;
see~\cite{roos2008robust}.

\subsection{Discrete boundary norm equivalence}
\label{sec:boundary-equiv}

The boundary residual in the stability estimate is measured in the
continuous $\HhalfBdry$ trace norm, whereas the loss only has access to the
discrete norm of Section~\ref{sec:boundary-disc}. The following theorem,
established in~\cite{singh2026cutpinn} for the level-set cut geometry, certifies
that the two are equivalent up to a term controlled by the boundary budget;
we use it without reproving it.

\begin{theorem}[Discrete $\HhalfBdry$ norm equivalence,
\protect{\cite[Theorem 3]{singh2026cutpinn}}]
\label{thm:boundary-norm-equiv}
Let $\dom \subset \R^{2}$ be a bounded domain whose boundary $\bdry$ is a
$C^{2}$ simple closed curve of length $L$ and maximum curvature
$\kappa_{\max}$, and let $Z = \{\bz_{1}, \ldots, \bz_{m}\} \subset \bdry$
be $m$ arc-length-equispaced points. For the trace model class
$G = \mathrm{Tr}\,U(B^{\bar{s}}_{\infty}(\Lp{2}(\dom)))$ with $\bar{s} > 1$,
\begin{equation}
\begin{array}{cc}
\norm{g}_{\HhalfBdry} \lesssim \norm{g}^{\ast}_{\HhalfBdry}
        + \norm{g}_{\mathrm{Tr}(B)} m^{-\bar{s}+1}  \\[6pt]
\norm{g}^{\ast}_{\HhalfBdry} \lesssim \norm{g}_{\HhalfBdry}
        + \norm{g}_{\mathrm{Tr}(B)} m^{-\bar{s}+1},
\end{array}
\end{equation}
with constants depending on $\kappa_{\max}$, $L$, $\bar{s}$, and the
polynomial degree $r > \max(\bar{s}, 1)$.
\end{theorem}

We invoke Theorem~\ref{thm:boundary-norm-equiv} as a black box; it supplies the
boundary half of the a priori bound below. The equispacing hypothesis is
moreover not essential: for boundary points drawn \emph{i.i.d.}\ uniformly
with respect to arc length, the discrete $\HhalfBdry$ seminorm concentrates
around its continuous counterpart at the Monte Carlo rate
$(m\delta)^{-1/2}$, with constants again controlled by the chord-arc
parameter~\cite[Proposition~4]{singh2026cutpinn}. This is the boundary
counterpart of Proposition~\ref{prop:lgamma-prob} below, so neither half of
the analysis is tied to a particular collocation layout.

\subsection{Discrete interior $\Lp{\gamma}$ comparison}
\label{sec:interior-comp}

We now turn to the interior residual, which is the part of the analysis
specific to the present work. The stability estimate controls it in
$\Hminusone$, and through the embedding $\Lp{\gamma}(\dom) \hookrightarrow
\Hminusone(\dom)$ it suffices to bound the continuous $\Lp{\gamma}$ norm of
the residual; the loss, however, only evaluates the residual at the interior
collocation points. The following proposition closes this gap on a cut
domain, showing that the discrete $\Lp{\gamma}$ functional approximates the
continuous $\Lp{\gamma}$ norm up to a rate $\mtil^{-\alpha_{\gamma}}$ in
which the curved boundary enters through a cut-cell floor.

\begin{proposition}[\textbf{Discrete $\Lp{\gamma}$-comparison on a cut domain}]
\label{prop:lgamma-cut}
Let $\dom$ be a bounded $C^{2}$ level-set domain, let $Q \supset \dom$ be an
axis-aligned bounding box, and let $X = G_{k,r} \cap \dom$ be the interior
points of the tensor-product grid $G_{k,r}$ of $Q$, with $\mtil := |X|$ and
$r > \max(s,1)$. Let $\gamma \in (1, 2]$. For every
$w \in B^{s}_{\gamma}(\Lp{\gamma}(\dom))$ with $s > 2/\gamma$,
\begin{equation}
\label{eq:lgamma-comparison}
    \bigl| \, \norm{w}_{\Lp{\gamma}(\dom)} - \norm{w}^{\ast}_{\Lp{\gamma}(\dom)}
    \, \bigr|
    \;\leq\;
    C_{\mathrm{int}} \, \norm{w}_{B^{s}_{\gamma}(\dom)} \, \mtil^{-\alpha_{\gamma}},
    \qquad
    \alpha_{\gamma} = \min\!\Bigl( \tfrac{s}{2} - \bigl(\tfrac{1}{\gamma} - \tfrac{1}{2}\bigr)_{+}, \ \tfrac{1}{2\gamma} \Bigr),
\end{equation}
where $\norm{w}^{\ast \gamma}_{\Lp{\gamma}(\dom)} :=
|\dom| \mtil^{-1} \sum_{i=1}^{\mtil} |w(\bx_{i})|^{\gamma}$ and $C_{\mathrm{int}}$
depends on $\dom$, $s$, $\gamma$, $r$, and the Stein extension constant.
\end{proposition}

\begin{proof}
The argument extends~\cite[Proposition~5]{singh2026cutpinn} from the
$\Lp{2}$ interior norm to $\Lp{\gamma}$, isolating the cut-cell contribution
as there. The exponent $\gamma$ enters in two places: the optimal-recovery
rate $\alpha$ of the box interpolant, and the $\gamma$-th-root passage from
the $\Lp{\gamma}$-power bound to the $\Lp{\gamma}$ norm, which sets the cut
floor.

\smallskip\noindent\textbf{Setup.} Let $\hat{w} = Ew$ be the Stein
extension, with $\norm{\hat w}_{B^{s}_{\gamma}(\R^{2})} \lesssim
\norm{w}_{B^{s}_{\gamma}(\dom)}$ and $\norm{\hat w}_{C(\bar Q)} \lesssim
\norm{w}_{B}$ via the embedding $B^{s}_{\gamma} \hookrightarrow C$, valid
since $s > 2/\gamma$. Let $P := S^{\ast}_{k}(\hat w) \in V^{r}(\mathcal{T}_{k})$
be the piecewise-polynomial interpolant of~\cite[Theorem~2.1]{bonito2025} on
the $N = (r \cdot 2^{k})^{2}$ nodes of $G_{k,r}$; their hypotheses
$0 < p \leq \infty$, $s > d/p$ hold here at $p = \gamma$, $d = 2$, so by the
$\Lp{\gamma}$ optimal-recovery estimate (case $p = \gamma$),
\begin{equation}
\label{eq:box-interp}
    \norm{\hat w - P}_{\Lp{\gamma}(Q)} \lesssim \norm{w}_{B}\, N^{-\alpha},
    \qquad
    \alpha = \tfrac{s}{2} - \bigl(\tfrac{1}{\gamma} - \tfrac{1}{2}\bigr)_{+}
           = \tfrac{s - 2/\gamma + 1}{2},
\end{equation}
the last equality because $1/\gamma \geq 1/2$ for $\gamma \leq 2$, and $\norm{P}_{C(\bar Q)} \lesssim
\norm{w}_{B}$ by the stability of $S^{\ast}_{k}$ (Lebesgue constant
$\Lambda_{r}$). The node and interior-point counts are comparable,
$N = |Q|\,h^{-2} r^{2}$ and $\mtil = |X| = |\dom|\,h^{-2}(1 + O(\mtil^{-1/2}))$
with $h$ the grid spacing, so $N \simeq \mtil$ with constant
$|Q|\,r^{2}/|\dom|$; we write $N^{-\alpha} \simeq \mtil^{-\alpha}$ and absorb
the constant into $C_{\mathrm{int}}$.

\smallskip\noindent\textbf{Cut-cell isolation.} Partition the cells meeting
$\bar\dom$ into interior cells (contained in $\dom$), with union
$\dom^{\mathrm{int}}$, and cut cells (meeting $\bdry$). Since $\bdry \in C^{2}$,
\begin{equation}
\label{eq:cut-area}
    |\dom \setminus \dom^{\mathrm{int}}| = O(\mtil^{-1/2}),
    \qquad
    \# \{\text{cut nodes}\} = O(\mtil^{1/2}).
\end{equation}
Both follow from the grid geometry: $h \simeq |\dom|^{1/2}\mtil^{-1/2}$, and a
$C^{2}$ boundary of length $L$ and bounded curvature $\kappa_{\max}$ meets
$O(L/h) = O(\mtil^{1/2})$ cells, so the cut region, a tube of width $O(h)$
about $\bdry$, has area $O(L\,h) = O(\mtil^{-1/2})$. On the interior cells the
discrete sum runs over $X^{\mathrm{int}} := X \cap \dom^{\mathrm{int}}$;
writing $\norm{w}^{\ast\gamma}_{\Lp{\gamma}(\dom^{\mathrm{int}})} :=
|\dom|\mtil^{-1}\sum_{\bx_i \in X^{\mathrm{int}}}|w(\bx_i)|^{\gamma}$ for its
contribution, the BDPS norm-equivalence in $\Lp{\gamma}$ gives, with
$P(\bx_{i}) = w(\bx_{i})$ on $X^{\mathrm{int}}$,
\begin{equation}
\label{eq:int-bound}
    \norm{P}^{\gamma}_{\Lp{\gamma}(\dom^{\mathrm{int}})}
    \leq \norm{w}^{\ast \gamma}_{\Lp{\gamma}(\dom^{\mathrm{int}})}
        \bigl(1 + O(\mtil^{-1/2})\bigr)
        + C \norm{w}^{\gamma}_{B}\, \mtil^{-\gamma\alpha}.
\end{equation}
On the cut region, H\"older with $\norm{P}_{C(\bar Q)} \lesssim \norm{w}_{B}$
and~\eqref{eq:cut-area} gives
\begin{equation}
\label{eq:cut-bound}
    \norm{P}^{\gamma}_{\Lp{\gamma}(\dom \setminus \dom^{\mathrm{int}})}
    \leq \norm{P}^{\gamma}_{C(\bar Q)}\,
         |\dom \setminus \dom^{\mathrm{int}}|
    \lesssim \norm{w}^{\gamma}_{B}\, \mtil^{-1/2}.
\end{equation}
Adding~\eqref{eq:int-bound} and~\eqref{eq:cut-bound}, and noting
$\norm{w}^{\ast\gamma}_{\Lp{\gamma}(\dom^{\mathrm{int}})} \leq
\norm{w}^{\ast\gamma}_{\Lp{\gamma}(\dom)}$, gives the $\Lp{\gamma}$-power
bound over all of $\dom$,
\begin{equation}
\label{eq:gamma-power-consolidated}
    \norm{P}^{\gamma}_{\Lp{\gamma}(\dom)}
    \leq \norm{w}^{\ast \gamma}_{\Lp{\gamma}(\dom)}
        \bigl(1 + O(\mtil^{-1/2})\bigr)
        + C \norm{w}^{\gamma}_{B}\, \mtil^{-\min(\gamma\alpha,\, 1/2)},
\end{equation}
the analogue of~\cite[Eq.~(39)]{singh2026cutpinn} with the power $2$ replaced
by $\gamma$. The area factor $\mtil^{-1/2}$ in the cut term
of~\eqref{eq:cut-bound} is purely geometric: it is the measure of the cut
tube in~\eqref{eq:cut-area}, independent of $\gamma$. What depends on
$\gamma$ is the conversion of~\eqref{eq:gamma-power-consolidated} into a
bound on the norm, and this sets the floor.

For $\gamma \geq 1$ the map $t \mapsto t^{1/\gamma}$ is concave on
$[0,\infty)$, vanishes at the origin, and is therefore subadditive:
$(a+b)^{1/\gamma} \leq a^{1/\gamma} + b^{1/\gamma}$ for $a, b \geq 0$.
Applying this to~\eqref{eq:gamma-power-consolidated} and using
$(1 + O(\mtil^{-1/2}))^{1/\gamma} = 1 + O(\mtil^{-1/2})$ for $\gamma \geq 1$,
\begin{equation}
\label{eq:cut-root}
    \norm{P}_{\Lp{\gamma}(\dom)}
    \leq \norm{w}^{\ast}_{\Lp{\gamma}(\dom)}\bigl(1 + O(\mtil^{-1/2})\bigr)
       + C \norm{w}_{B}\, \mtil^{-\min(\alpha,\, 1/(2\gamma))},
\end{equation}
since $(\mtil^{-1/2})^{1/\gamma} = \mtil^{-1/(2\gamma)}$ and
$(\mtil^{-\gamma\alpha})^{1/\gamma} = \mtil^{-\alpha}$, the minimum being
preserved by the monotone root.

Two consequences are worth isolating. At $\gamma = 2$ the floor is
$1/(2\gamma) = 1/4$, the cut-cell ceiling $\min(\alpha, 1/4)$ of
\cite[Proposition~5]{singh2026cutpinn}: the $\Lp{2}$ case is the present one
with the $\gamma$-th root specialised to a square root. For
$\gamma < 2$ one has $1/(2\gamma) > 1/4$, so the cut region
degrades the recovery rate \emph{less} in $\Lp{\gamma}$ than in $\Lp{2}$; the
same choice of $\gamma$ near $1$ that down-weights the boundary layer also
relaxes the geometric penalty. As $\gamma \to 1^{+}$ the floor tends to
$1/2$, the cut-free Monte Carlo rate of~\cite[Proposition~6]{singh2026cutpinn};
the interpolation rate $\alpha$ governs once $\alpha \leq 1/(2\gamma)$, i.e.
$s \leq 3/\gamma - 1$, and the floor governs above it.

The floor $1/(2\gamma)$ is sharp. Let $w$ be a unit-height $C^{\infty}$ bump
supported in $\dom \setminus \dom^{\mathrm{int}}$ away from every
interior node, a region of area $\simeq \mtil^{-1/2}$
by~\eqref{eq:cut-area}, with $\norm{w}_{B} \simeq 1$. Then
$\norm{w}^{\gamma}_{\Lp{\gamma}(\dom)} \simeq \mtil^{-1/2}$ while
$\norm{w}^{\ast\gamma}_{\Lp{\gamma}(\dom)} = o(\mtil^{-1/2})$, whence
$\bigl|\norm{w}_{\Lp{\gamma}(\dom)} - \norm{w}^{\ast}_{\Lp{\gamma}(\dom)}\bigr|
\simeq \mtil^{-1/(2\gamma)}$, so no exponent exceeding $1/(2\gamma)$ holds
uniformly over the model class. A bump on a single cut cell, of area
$\simeq \mtil^{-1}$, gives only $\mtil^{-1/\gamma}$ and does not realise the
floor.

\bigskip
\noindent
\textbf{Assembly.} The uniform bound $\norm{w}^{\ast}_{\Lp{\gamma}(\dom)}
\leq |\dom|^{1/\gamma}\norm{w}_{C(\bar\dom)} \lesssim \norm{w}_{B}$ absorbs
the ratio correction in~\eqref{eq:cut-root}: the factor
$(1 + O(\mtil^{-1/2}))$ contributes $\lesssim \norm{w}_{B}\,\mtil^{-1/2}$,
dominated by the existing error term since $1/(2\gamma) \leq 1/2$. Thus
\begin{equation}
\label{eq:P-norm-clean}
    \bigl| \, \norm{P}_{\Lp{\gamma}(\dom)} - \norm{w}^{\ast}_{\Lp{\gamma}(\dom)}
    \, \bigr|
    \leq C\,\norm{w}_{B}\,\mtil^{-\alpha_{\gamma}},
    \qquad \alpha_{\gamma} := \min\bigl(\alpha,\, 1/(2\gamma)\bigr),
\end{equation}
the lower bound on $\norm{P}_{\Lp{\gamma}(\dom)}$ coming from the same
estimates~\eqref{eq:int-bound}--\eqref{eq:cut-bound} read in reverse over
$X = X^{\mathrm{int}} \cup (X\setminus X^{\mathrm{int}})$, where the cut sum
carries $O(\mtil^{1/2})$ nodes of weight $|\dom|/\mtil$ and so contributes
the same $\mtil^{-1/2}$ in the $\gamma$-power. Finally, the interpolant
approximates $w$ itself at the faster rate: by~\eqref{eq:box-interp} and
$w = \hat w$ on $\dom \subset Q$,
\begin{equation}
\label{eq:wP-norm-clean}
    \norm{w - P}_{\Lp{\gamma}(\dom)} \leq \norm{\hat w - P}_{\Lp{\gamma}(Q)}
    \lesssim \norm{w}_{B}\,\mtil^{-\alpha}
    \leq \norm{w}_{B}\,\mtil^{-\alpha_{\gamma}}
\end{equation}
since $\alpha \geq \alpha_{\gamma}$. Next, we have
\[
\bigl|\,\norm{w}_{\Lp{\gamma}(\dom)} - \norm{w}^{\ast}_{\Lp{\gamma}(\dom)}\,\bigr|
\leq \norm{w - P}_{\Lp{\gamma}(\dom)}
+ \bigl|\,\norm{P}_{\Lp{\gamma}(\dom)} - \norm{w}^{\ast}_{\Lp{\gamma}(\dom)}\,\bigr|.
\]
Combining \eqref{eq:wP-norm-clean} with~\eqref{eq:P-norm-clean}, we
obtain~\eqref{eq:lgamma-comparison}, which completes the proof.
\end{proof}

\begin{remark}[Inherited rate, new floor, and the sampled interior set]
Of the two exponents in $\alpha_{\gamma} = \min(\alpha, 1/(2\gamma))$, the
interpolation rate $\alpha = (s - 2/\gamma + 1)/2$ is the $\Lp{\gamma}$
instance ($p = \gamma$) of the optimal-recovery estimate
\cite[Theorem~2.1]{bonito2025}, used as a black box; the cut floor
$1/(2\gamma)$ has no flat-domain precedent, entering only through the
$\gamma$-th-root passage of the proof and relaxing toward $1/2$ as
$\gamma \to 1$, so the exponent improves precisely in the regime where
$\Lp{\gamma}$ is used to suppress the layer. At $\gamma = 2$ the floor is
$1/4$, recovering~\cite[Proposition~5]{singh2026cutpinn}. The proposition is
stated for the deterministic grid $X = G_{k,r} \cap \dom$, on which the
cut-cell partition is defined; the \emph{i.i.d.}\ rejection sampling used in
\Cref{sec:numerics} is covered by the companion Proposition~\ref{prop:lgamma-prob},
which gives the same comparison at the Monte Carlo rate $\mtil^{-1/2}$, no
slower than the floor and independent of the curvature, and reduces at
$\gamma = 2$ to the $\Lp{2}$ estimate~\cite[Proposition~6]{singh2026cutpinn};
like that result it is not invoked in Theorem~\ref{thm:apriori-lgamma}.
\end{remark}

\begin{proposition}[\textbf{Probabilistic $\Lp{\gamma}$ comparison under \emph{i.i.d.}\ uniform sampling}]
\label{prop:lgamma-prob}
Let $\dom \subset \R^{2}$ be a bounded domain and let
$X = \{\bx_{1}, \dots, \bx_{\mtil}\} \subset \dom$ be drawn \emph{i.i.d.}\ uniformly
on $\dom$. Set
$\norm{w}_{\Lp{\gamma}(\dom)}^{\ast\gamma} := |\dom|\,\mtil^{-1}
\sum_{i=1}^{\mtil} |w(\bx_{i})|^{\gamma}$ as in Proposition~\ref{prop:lgamma-cut}. For
every $w \in \Lp{\infty}(\dom)$ with $\norm{w}_{\Lp{\gamma}(\dom)} > 0$,
every $\gamma \geq 1$, and every $\delta \in (0,1)$, with probability at
least $1 - \delta$,
\begin{equation}
\label{eq:lgamma-prob}
    \bigl| \, \norm{w}_{\Lp{\gamma}(\dom)} - \norm{w}_{\Lp{\gamma}(\dom)}^{\ast} \, \bigr|
    \;\leq\;
    \frac{2^{\gamma-1}\,|\dom|^{1/2}\,\norm{w}_{\Lp{\infty}(\dom)}^{\gamma/2}}
         {\gamma\,\norm{w}_{\Lp{\gamma}(\dom)}^{\gamma/2 - 1}}\,
    (\mtil\,\delta)^{-1/2}.
\end{equation}
For $w \in U(B^{s}_{\gamma}(\Lp{\gamma}(\dom)))$ with $s > 2/\gamma$, the embedding
$B^{s}_{\gamma} \hookrightarrow C(\bar{\dom})$ bounds
$\norm{w}_{\Lp{\infty}}$ by $C(\dom, s, \gamma)\,\norm{w}_{B}$, yielding the
Monte Carlo rate $\mtil^{-1/2}$ uniformly over the model class.
\end{proposition}

\begin{proof}
The argument mirrors~\cite[Proposition~6]{singh2026cutpinn}, with the square
replaced by the $\gamma$-th power. Let $Y_{i} := |w(\bx_{i})|^{\gamma}$ and
$S_{\mtil} := \mtil^{-1} \sum_{i=1}^{\mtil} Y_{i}$. Since the $\bx_{i}$ are
\emph{i.i.d.}\ uniform on $\dom$,
\[
    \mathbb{E}[Y_{1}]
    = \frac{1}{|\dom|} \int_{\dom} |w|^{\gamma}
    = \frac{\norm{w}_{\Lp{\gamma}(\dom)}^{\gamma}}{|\dom|},
    \qquad
    \mathrm{Var}(Y_{1})
    \leq \mathbb{E}[Y_{1}^{2}]
    \leq \frac{\norm{w}_{\Lp{\infty}(\dom)}^{\gamma}\,
               \norm{w}_{\Lp{\gamma}(\dom)}^{\gamma}}{|\dom|},
\]
the second bound using
$\int_{\dom} |w|^{2\gamma} \leq \norm{w}_{\Lp{\infty}}^{\gamma}
\int_{\dom} |w|^{\gamma}$. Hence, we get
\[
\mathrm{Var}(S_{\mtil}) = \mathrm{Var}(Y_{1})/\mtil \leq
\norm{w}_{\Lp{\infty}}^{\gamma}\,\norm{w}_{\Lp{\gamma}(\dom)}^{\gamma}/
(|\dom|\,\mtil).
\]
Therefore, Chebyshev's inequality yields, with probability at
least $1 - \delta$,
$$|S_{\mtil} - \mathbb{E}[Y_{1}]| \leq
\sqrt{\mathrm{Var}(S_{\mtil})/\delta} \leq
\norm{w}_{\Lp{\infty}}^{\gamma/2}\,\norm{w}_{\Lp{\gamma}(\dom)}^{\gamma/2}/
\sqrt{|\dom|\,\mtil\,\delta}.$$ 
Multiplying through by $|\dom|$,
\begin{equation}
\label{eq:lgamma-prob-pow}
    \bigl| \, \norm{w}_{\Lp{\gamma}(\dom)}^{\ast\gamma}
           - \norm{w}_{\Lp{\gamma}(\dom)}^{\gamma} \, \bigr|
    \;\leq\;
    |\dom|^{1/2}\,\norm{w}_{\Lp{\infty}(\dom)}^{\gamma/2}\,
    \norm{w}_{\Lp{\gamma}(\dom)}^{\gamma/2}\,(\mtil\,\delta)^{-1/2}
    \;=:\; K\,(\mtil\,\delta)^{-1/2},
\end{equation}
which holds unconditionally on the probability-$(1-\delta)$ event. To pass
from the $\gamma$-power to the norm, write $a := \norm{w}_{\Lp{\gamma}(\dom)}^{\ast}$
and $b := \norm{w}_{\Lp{\gamma}(\dom)} > 0$. Two consequences
of~\eqref{eq:lgamma-prob-pow} are used in turn. First, since
$a^{\gamma} \geq b^{\gamma} - K(\mtil\delta)^{-1/2}$, the budget threshold
$\mtil \geq K^{2}\,\delta^{-1}(1 - 2^{-\gamma})^{-2}\,b^{-2\gamma}$ forces
$a^{\gamma} \geq 2^{-\gamma} b^{\gamma}$, that is $a \geq b/2$; this is a
deterministic consequence of the bound, not a further probabilistic
assumption. Second, on this range the mean value theorem applied to
$t \mapsto t^{\gamma}$ gives $|a^{\gamma} - b^{\gamma}| \geq
\gamma\,\min(a,b)^{\gamma-1}|a - b| \geq \gamma\,(b/2)^{\gamma-1}|a-b|$, so
\[
    |a - b| \leq \frac{2^{\gamma-1}}{\gamma\,b^{\gamma-1}}\,
                |a^{\gamma} - b^{\gamma}|.
\]
Substituting~\eqref{eq:lgamma-prob-pow} gives~\eqref{eq:lgamma-prob}. At
$\gamma = 2$ the threshold is unnecessary: the identity
$|a-b|(a+b) = |a^{2}-b^{2}|$ gives $|a-b| \leq |a^{2}-b^{2}|/b$ for all
$a \geq 0$, $b > 0$, recovering~\cite[Proposition~6]{singh2026cutpinn} with
the $\norm{w}_{\Lp{\gamma}}$ factors cancelling completely. The Besov special
case follows from $\norm{w}_{\Lp{\infty}(\dom)} \leq
C(\dom, s, \gamma)\,\norm{w}_{B}$ via the embedding for $s > 2/\gamma$,
yielding the Monte Carlo rate $\mtil^{-1/2}$.
\end{proof}

\subsection{Main a priori bound}
\label{sec:apriori}

Combining the continuous stability of Section~\ref{sec:cont-stability} with
the two discrete comparisons gives the main result. The following theorem
bounds the $\Hone$ error of an arbitrary trial network by the computable
loss $\Lgam(v)$ and a discretisation remainder that decays in both
collocation budgets.

\begin{theorem}[\textbf{A priori bound for the $\Lp{\gamma}$ consistent CutPINN}]
\label{thm:apriori-lgamma}
Let $\dom \subset \R^{2}$ be a bounded $C^{2}$ level-set domain and let the
coefficients of~\eqref{eq:cdr} satisfy the assumptions of
\Cref{sec:cdr-pde}. Let $X$ and $Z$ be interior and boundary collocation
sets as in \Cref{sec:collocation} with $|X| = \mtil$ and $|Z| = m$.
Let $v \in \Hone(\dom)$ with $\Delta v$ continuous on $\bar{\dom}$ and
suppose $v \in U(B^{\bar{s}}_{\gamma}(\Lp{\gamma}(\dom)))$ with $\bar{s} >
2/\gamma$ and $\gamma \in (1, 2]$, and write
$\norm{v}_{U} := \max\{\norm{\Lop v}_{B}, \norm{\mathrm{Tr}(v)}_{\mathrm{Tr}(B)}\}$.
Then for every $\eps > 0$,
\begin{equation}
\label{eq:apriori-lgamma}
    \norm{u - v}_{\Hone(\dom)}
    \;\leq\;
    C(\eps^{-1}, \kappa_{\max}, L, \norm{\bb}_{\infty}, c)\,\tfrac{\gamma}{\gamma-1}
    \Bigl[ \Lgam(v)^{1/2} + \bigl(1 + \norm{v}_{U}\bigr) R(\mtil, m) \Bigr],
\end{equation}
where
$R(\mtil, m) = \mtil^{-\alpha_{\gamma}} + m^{-\bar{s}+1}$ and
$\alpha_{\gamma} = \min\!\bigl(\tfrac{\bar{s}}{2} - \bigl(\tfrac{1}{\gamma} - \tfrac{1}{2}\bigr)_{+}, \ 1/(2\gamma) \bigr)$.
\end{theorem}

\begin{remark}[The prescribed $\gamma$ and the $\Lp{2}$ case]
The bound holds for every $\gamma \in (1,2]$, with the prefactor
$\gamma/(\gamma-1)$ measuring the cost of the $\Lp{\gamma} \hookrightarrow
\Hminusone$ embedding. The proposed method takes $\gamma = 1 + 1/\log\mtil$,
at which $\gamma/(\gamma-1) = 1 + \log\mtil$ and the interior norm carries no
free parameter; the consistent $\Lp{2}$ loss is the endpoint $\gamma = 2$,
where the prefactor is the constant $2$ and the embedding is non-degenerate.
The intermediate $\gamma$ trades a mild $\log\mtil$ growth in the constant
for the layer suppression established in Remark~\ref{rem:l2-fails}, which is the
exchange the convection-dominated regime rewards.
\end{remark}

\begin{proof}
The argument follows the four-step template
of~\cite[Theorem~7.2]{bonito2025}, \cite[Theorem~7]{singh2026cutpinn}, with
the operator and the interior norm adapted to the convection-diffusion
setting.

\vskip 2mm
\noindent
\textbf{Step 1: Continuous stability.} The coercivity assumption
$c - \tfrac12 \nabla\!\cdot\!\bb \geq \kappa_{0} > 0$ makes the weak form
of~\eqref{eq:cdr} coercive in the energy norm
$\enorm{w}^{2} := \eps\norm{\nabla w}_{\Lp{2}}^{2}
+ \kappa_{0}\norm{w}_{\Lp{2}}^{2}$, and converting to $\Hone$ through
$\norm{\nabla w}_{\Lp{2}} \leq \eps^{-1/2}\enorm{w}$ gives the stability
estimate~\eqref{eq:continuous-stability}. The two conversions compound: the
duality pairing $\langle \Lop w, w\rangle \leq \norm{\Lop w}_{\Hminusone}
\norm{w}_{\Hone}$ and the energy-to-$\Hone$ passage each contribute
$\eps^{-1/2}$, so $C_{\mathrm{stab}} \sim \eps^{-1}$;
see~\cite{roos2008robust}. Applied to $w = u - v$, this
reduces the $\Hone$ error to an $\Hminusone$ interior residual
$\norm{f - \Lop v}_{\Hminusone(\dom)}$ and an $\HhalfBdry$ boundary residual
$\norm{g - v}_{\HhalfBdry}$.

\vskip 2mm
\noindent
\textbf{Step 2: Interior embedding.} The interior residual is controlled in
$\Hminusone$ by its $\Lp{\gamma}$ norm through the Sobolev embedding
$\Lp{\gamma}(\dom) \hookrightarrow \Hminusone(\dom)$, applied to
$w = f - \Lop v$. In $d = 2$ this embedding holds for every $\gamma > 1$;
the endpoint $\gamma = 1$ fails, since $\Hone(\dom)$ does not embed into
$\Lp{\infty}(\dom)$, and correspondingly the embedding constant degenerates
as $\gamma \to 1^{+}$, bounded by a multiple of $\gamma/(\gamma - 1)$,
which is the logarithmic factor identified for $p \leq 1$
in the $d = 2$ analysis of~\cite{bonito2025}. For any $\gamma \in (1,2]$ the
embedding contributes a constant factor of order $\gamma/(\gamma - 1)$, finite and
independent of $\eps$; at the prescribed $\gamma = 1 + 1/\log\mtil$, which
approaches the endpoint as the interior budget grows, this equals
$1 + \log\mtil$, the $O(\log\mtil)$ factor. The
same optimal-recovery treatment of a convective residual is used at the
$\Lp{2}$ level by~\cite{mishra2026structure}; the $\Lp{\gamma}$ exponent and the
resulting $\gamma/(\gamma-1)$ factor are what the convection-dominated regime
adds here.

\vskip 2mm
\noindent
\textbf{Step 3: Discrete interior and boundary comparison.} Proposition~\ref{prop:lgamma-cut} with
$w = f - \Lop v$ and $\norm{f}_{B} \leq 1$ replaces the continuous
$\Lp{\gamma}$ residual norm by its discrete counterpart, with error
$(1+\norm{v}_{U})\,\mtil^{-\alpha_{\gamma}}$.
Theorem~\ref{thm:boundary-norm-equiv} replaces the continuous $\HhalfBdry$
boundary residual norm by its discrete counterpart, with error
$(1+\norm{v}_{U})\,m^{-\bar{s}+1}$; this is where the geometry, through
$\kappa_{\max}$ and the chord-arc constant, enters.

\vskip 2mm
\noindent
\textbf{Assembly.} Combining Steps~1--3, applying the elementary inequality
$a + b \leq \sqrt{2}\,(a^{2}+b^{2})^{1/2}$ to the two discrete residual
terms and identifying the sum of their squares with $\Lgam(v)$, whose interior term
carries the matching outer power $2/\gamma$ of Definition~\ref{def:losses},
yields~\eqref{eq:apriori-lgamma}. The $\gamma/(\gamma-1)$ embedding constant of
Step~2 multiplies the entire interior contribution, both the computable loss
$\Lgam(v)^{1/2}$ and its recovery remainder $\mtil^{-\alpha_{\gamma}}$; we
display it as a single prefactor on the whole bracket, which only weakens
the boundary term $m^{-\bar{s}+1}$ by the same harmless $\gamma/(\gamma-1)$
factor (logarithmic at the prescribed $\gamma$).
\end{proof}

\begin{remark}[Reading the bound: $\eps$-dependence and the predicted rate]
\label{rem:rate}
The constant $C$ in~\eqref{eq:apriori-lgamma} depends on $\eps^{-1}$,
matching the $\Hone$ stability constant of the continuous
problem~\cite{roos2008robust}, so the discretisation introduces no loss in
$\eps$ beyond what the continuous theory itself carries; the route to an
$\eps$-uniform version, through a streamline-diffusion stability analysis
on the cut domain, is outlined in \Cref{sec:conclusion}. The bound is
designed to be read at fixed $\eps$: the model-class norm $\norm{v}_{U}$
involves $\norm{\Lop v}_{B}$ and hence the term $\eps\Delta v$, which grows
as the network resolves the $O(\eps)$ layer, and the experiments
accordingly fix $\eps$ and vary the initialisation, which is precisely the
comparison the bound supports.
At the budget $\mtil = 1600$ one has $\gamma \approx 1.136$, so for
the smooth data of Section~\ref{sec:exp-setup} the interior exponent is
governed by the cut floor, $\alpha_{\gamma} = 1/(2\gamma) \approx 0.44$,
rather than the interpolation rate, and the boundary term contributes
$m^{-\bar{s}+1}$. The experiments below fix $(\mtil, m)$ and vary $\eps$, so
they probe the robustness of training as the layer sharpens, not the decay
of the bound in the budget; a direct measurement of the $\mtil^{-\alpha_{\gamma}}$
rate would require a separate budget-refinement study at fixed $\eps$.
\end{remark}

\section{Numerical Experiments}
\label{sec:numerics}

All experiments use a PyTorch~\cite{paszke2019pytorch} implementation
running on a single Intel Core Ultra 7 155H CPU, with each training run
completing in well under a minute. Code, data, and scripts to reproduce
every figure and table are available at
\url{https://github.com/maneeshkrsingh/consistent-cutpinn}.

\subsection{Setup}
\label{sec:exp-setup}

\textbf{Manufactured solution.}
The exact solution is constructed to carry an explicit boundary layer
aligned with the streamline direction. Fixing
$\alpha = \pi/3$, $\bb = (\cos\alpha, \sin\alpha)$, $c = 1$, and the
streamline and cross-streamline coordinates
$s_{\mathrm{c}} = x_{1}\cos\alpha + x_{2}\sin\alpha$,
$t_{\mathrm{c}} = -x_{1}\sin\alpha + x_{2}\cos\alpha$
normalised to $[0, 1]$, define
\begin{equation}
\label{eq:manufactured}
    u(x_{1}, x_{2}) = \sin\bigl(\pi t_{\mathrm{c,norm}}\bigr)
                      \cdot \psi_{\eps}\bigl(s_{\mathrm{c,norm}}\bigr),
    \quad
    \psi_{\eps}(\sigma) = \sigma
        - \frac{e^{(\sigma-1)/\eps} - e^{-1/\eps}}{1 - e^{-1/\eps}}.
\end{equation}
The factor $\psi_{\eps}$ is the classical 1D convection-diffusion layer
profile, smooth on $[0, 1 - O(\eps)]$ with an $O(\eps)$-wide boundary
layer at $\sigma = 1$. The data $f$ and $g$ follow from
inserting~\eqref{eq:manufactured} into~\eqref{eq:cdr}.

\vskip 2mm

\textbf{Domains:}
We use two test domains, one flat and one curved domain:
\begin{itemize}
    \item \textbf{Rectangle.} $\dom = (0,1)^{2}$. Flat boundary.
    \item \textbf{Disk.} $\bphi(\bx) = \norm{\bx - (0.5, 0.5)} - 0.4$,
          embedded in $[0,1]^{2}$. Constant curvature $\kappa = 2.5$; the
          curved test geometry.
\end{itemize}

\textbf{Computational budget:}
Interior budget $\mtil = 1600$, boundary budget $m = 40$, fixed throughout
the experiments. At this budget,
$\gamma = 1 + 1/\log(1600) \approx 1.136$. \\

\textbf{Error measurement:}
We estimate the $\Hone$ error norm by Monte Carlo quadrature over
$10{,}000$ points drawn independently by rejection sampling, evaluating
$\nabla v_{\theta}$ by automatic differentiation of the trained network. The
comparison at $s = 6$ is run over 20 independent seeds; at
$s = 2$ and $s = 4$, where every run is stable, five seeds suffice. \\

We report the median over non-diverged seeds together with the
divergence count and a Wilson score interval (SI) on the rate. A run
is counted diverged if its $\Hone$ error exceeds $2 \times 10^{-2}$; this
threshold sits cleanly above the converged-run band (typical $\Hone$ at
$s = 6$ is $\sim\!10^{-3}$) and below the failure cluster ($\sim\!10^{-1}$),
and the reported rates are insensitive to varying the threshold across
$[2 \times 10^{-2}, 10^{-1}]$. The median statistic is used
because in the convection-dominated regime the seed distribution is
bimodal and the mean is dominated by occasional diverged runs.
\Cref{fig:setup} illustrates the experimental setup on the two domains.

\begin{figure}[!htb]
    \centering
    \includegraphics[width=0.895\linewidth]{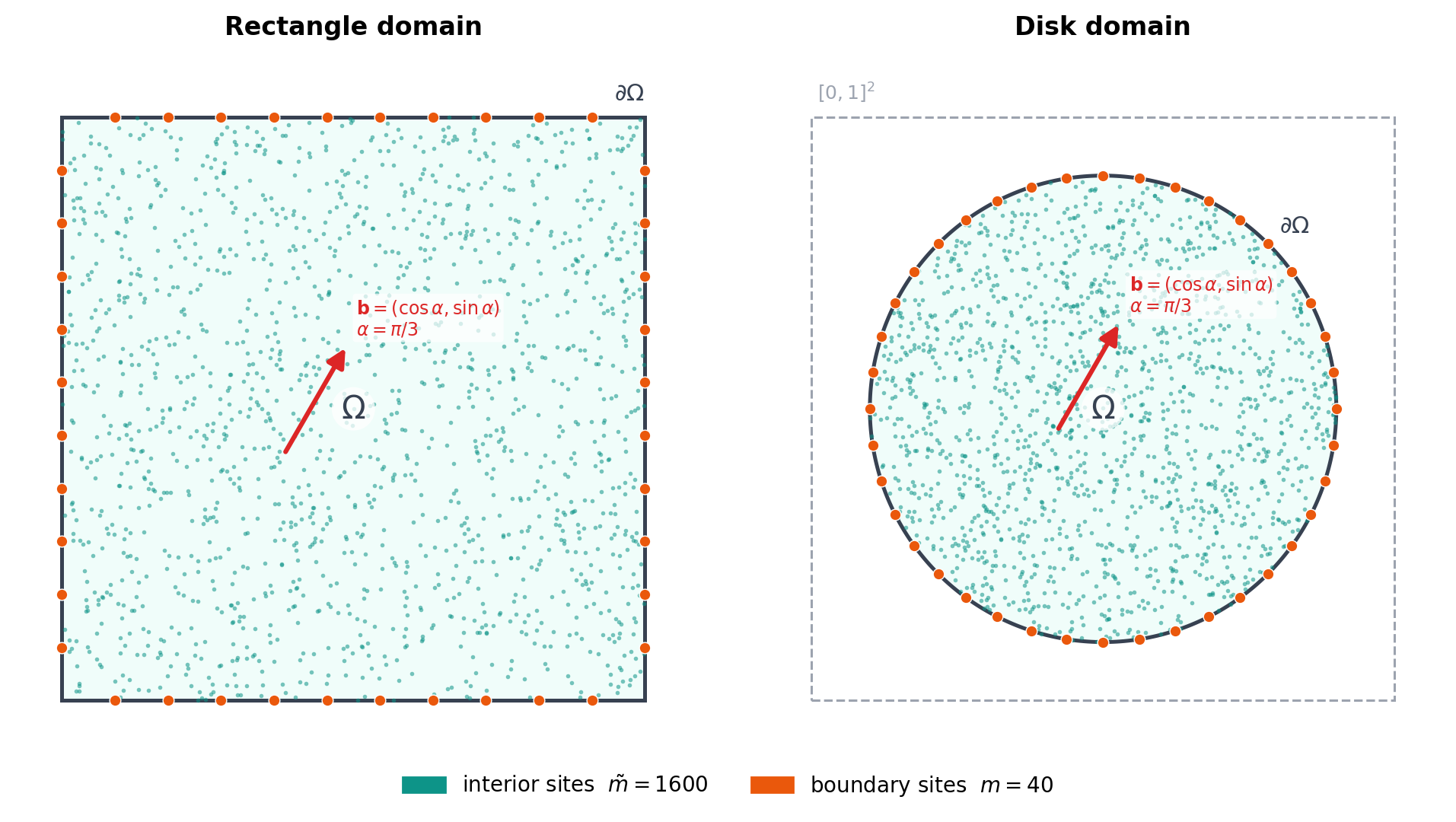}
    \caption{Experimental setup. Left: rectangle $\dom=(0,1)^{2}$; right:
    disk $\dom=\{\norm{\bx-(0.5,0.5)}<0.4\}$. Teal: $\mtil=1600$ interior
    points $X$; orange: $m=40$ boundary points $Z$. Red arrow: convection
    $\bb=(\cos\alpha,\sin\alpha)$, $\alpha=\pi/3$.}
    \label{fig:setup}
\end{figure}

\subsection{Experiment 1: Convergence in \texorpdfstring{$\eps$}{epsilon} on the rectangle}
\label{sec:exp-rectangle}

\Cref{tab:rectangle} reports the $\Hone$ error on the rectangle for the
three losses across $\eps = 2^{-s}$, $s \in \{2, 4, 6\}$.

\begin{table}[t]
    \centering
    \caption{$\Hone$ error on the rectangle $(0,1)^{2}$ across
    $\eps = 2^{-s}$, reported as the \emph{median over non-diverged seeds}
    (diverged: $\Hone > 2 \times 10^{-2}$; the count $d/n$ is listed in the
    lower block, with the Wilson SI on the rate at $s = 6$). Five
    seeds at $s = 2, 4$ (where training is uniformly stable); 20 seeds at
    $s = 6$. All three losses use the identical AdamW$\to$L-BFGS schedule of
    \Cref{sec:network}, so the comparison isolates the effect of the loss.
    The rectangle gives a complete threshold-robust separation at
    $s = 6$: standard PINN fails on every seed while $\Lgam$ never diverges.
    Bold marks the proposed $\Lgam$ method; ``--'' denotes all seeds
    diverged.}
    \label{tab:rectangle}
    \begin{tabular}{lccc}
        \toprule
        Loss & $\eps = 2^{-2}$
             & $\eps = 2^{-4}$
             & $\eps = 2^{-6}$ \\
        \midrule
        \multicolumn{4}{l}{\emph{median $\Hone$ error (non-diverged seeds)}} \\
        $\Lpinn$ (standard PINN)         & $5.10 \times 10^{-3}$
                                         & $6.21 \times 10^{-3}$
                                         & -- \\
        $\Lcons$ ($\Lp{2}$ + $\Hhalf$)   & $6.61 \times 10^{-4}$
                                         & $8.65 \times 10^{-4}$
                                         & $8.67 \times 10^{-3}$ \\
        $\Lgam$ ($\Lp{\gamma}$ + $\Hhalf$) & $\mathbf{7.86 \times 10^{-4}}$
                                         & $\mathbf{6.62 \times 10^{-4}}$
                                         & $\mathbf{4.96 \times 10^{-3}}$ \\
        \midrule
        \multicolumn{4}{l}{\emph{divergence count and rate (Wilson SI at $s = 6$)}} \\
        $\Lpinn$                         & $0/5$ & $2/5$
                                         & $20/20 = 100\%~[84,100]$ \\
        $\Lcons$                         & $0/5$ & $0/5$
                                         & $2/20 = 10\%~[3,30]$ \\
        $\Lgam$                          & $\mathbf{0/5}$ & $\mathbf{0/5}$
                                         & $\mathbf{0/20 = 0\%~[0,16]}$ \\
        \bottomrule
    \end{tabular}
\end{table}

For runs that converge, the three losses deliver comparable accuracy on
the rectangle: at $s = 2$ and $s = 4$ all sit near $10^{-3}$ in median
$\Hone$. The separation appears in \emph{reliability} as $\eps$ decreases.
At $s = 6$ the contrast is sharp. Standard PINN diverges on every seed
(20/20 above the threshold), the consistent loss with $\Lp{2}$ interior
diverges on two of twenty seeds, and the $\Lp{\gamma}$ interior never
diverges (0/20). The 20/20 versus 0/20 contrast is robust across any
threshold between $2\times 10^{-2}$ and $10^{-1}$, so it reflects a
training instability of standard PINN rather than a threshold artefact.
Most standard-PINN seeds settle near $\Hone \approx 4\times 10^{-2}$, a
qualitatively wrong solution, and none reaches an accurate fit. On the
seeds where $\Lcons$ converges, its accuracy is close to that of $\Lgam$;
it is the $\Lp{\gamma}$ interior that makes training dependable as the
layer sharpens.

\subsection{Experiment 2: Convergence in
\texorpdfstring{$\eps$}{epsilon} on the disk}
\label{sec:exp-disk}

\Cref{tab:disk} reports the same experiment on the disk, the curved
geometry.

\begin{table}[t]
    \centering
    \caption{$\Hone$ error on the disk across $\eps = 2^{-s}$, reported as
    the \emph{median over non-diverged seeds} (diverged: $\Hone > 2 \times
    10^{-2}$; the count $d/n$ is in the lower block, with Wilson SI on
    the rate at $s = 6$). Five seeds at $s = 2, 4$ (uniformly stable); 20
    seeds at $s = 6$. On the disk the divergence ordering
    $\Lgam \le \Lcons \le \Lpinn$ holds across the range, but the rates are
    low for all three losses and the 95\% intervals overlap at $s = 6$;
    the separation is suggestive on the disk but not statistically resolved
    by 20 seeds. Bold marks the proposed $\Lgam$ method.}
    \label{tab:disk}
    \begin{tabular}{lccc}
        \toprule
        Loss & $\eps = 2^{-2}$
             & $\eps = 2^{-4}$
             & $\eps = 2^{-6}$ \\
        \midrule
        \multicolumn{4}{l}{\emph{median $\Hone$ error (non-diverged seeds)}} \\
        $\Lpinn$                         & $3.30 \times 10^{-3}$
                                         & $2.11 \times 10^{-3}$
                                         & $1.21 \times 10^{-3}$ \\
        $\Lcons$                         & $6.28 \times 10^{-4}$
                                         & $3.12 \times 10^{-4}$
                                         & $7.26 \times 10^{-4}$ \\
        $\Lgam$                          & $\mathbf{6.39 \times 10^{-4}}$
                                         & $\mathbf{2.86 \times 10^{-4}}$
                                         & $\mathbf{5.02 \times 10^{-4}}$ \\
        \midrule
        \multicolumn{4}{l}{\emph{divergence count and rate (Wilson SI at $s = 6$)}} \\
        $\Lpinn$                         & $0/5$ & $1/5$
                                         & $2/20 = 10\%~[3,30]$ \\
        $\Lcons$                         & $0/5$ & $0/5$
                                         & $2/20 = 10\%~[3,30]$ \\
        $\Lgam$                          & $\mathbf{0/5}$ & $\mathbf{0/5}$
                                         & $\mathbf{1/20 = 5\%~[1,24]}$ \\
        \bottomrule
    \end{tabular}
\end{table}

The disk shows a less clear-cut separation than the rectangle. For seeds
that converge, the three losses reach comparable median $\Hone$ error
across $s \in \{2, 4, 6\}$: $\Lgam$ and $\Lcons$ both lie near
$5\times 10^{-4}$ at $s = 6$, and $\Lpinn$ is within a factor of three on
its surviving seeds. The divergence rates at $s = 6$ follow the same
ordering as on the rectangle ($\Lgam$ 1/20, $\Lcons$ 2/20, $\Lpinn$ 2/20),
but they are uniformly low and their Wilson 95\% intervals overlap heavily.
The disk result is therefore consistent with the rectangle finding and
supports it, but is not statistically resolved at the present seed count.
We report it as suggestive corroboration, with the rectangle carrying the
threshold-robust reliability claim. \Cref{fig:headline} visualises the
per-seed spread on both domains.

\begin{figure}[t]
    \centering
    \includegraphics[width=0.995\linewidth]{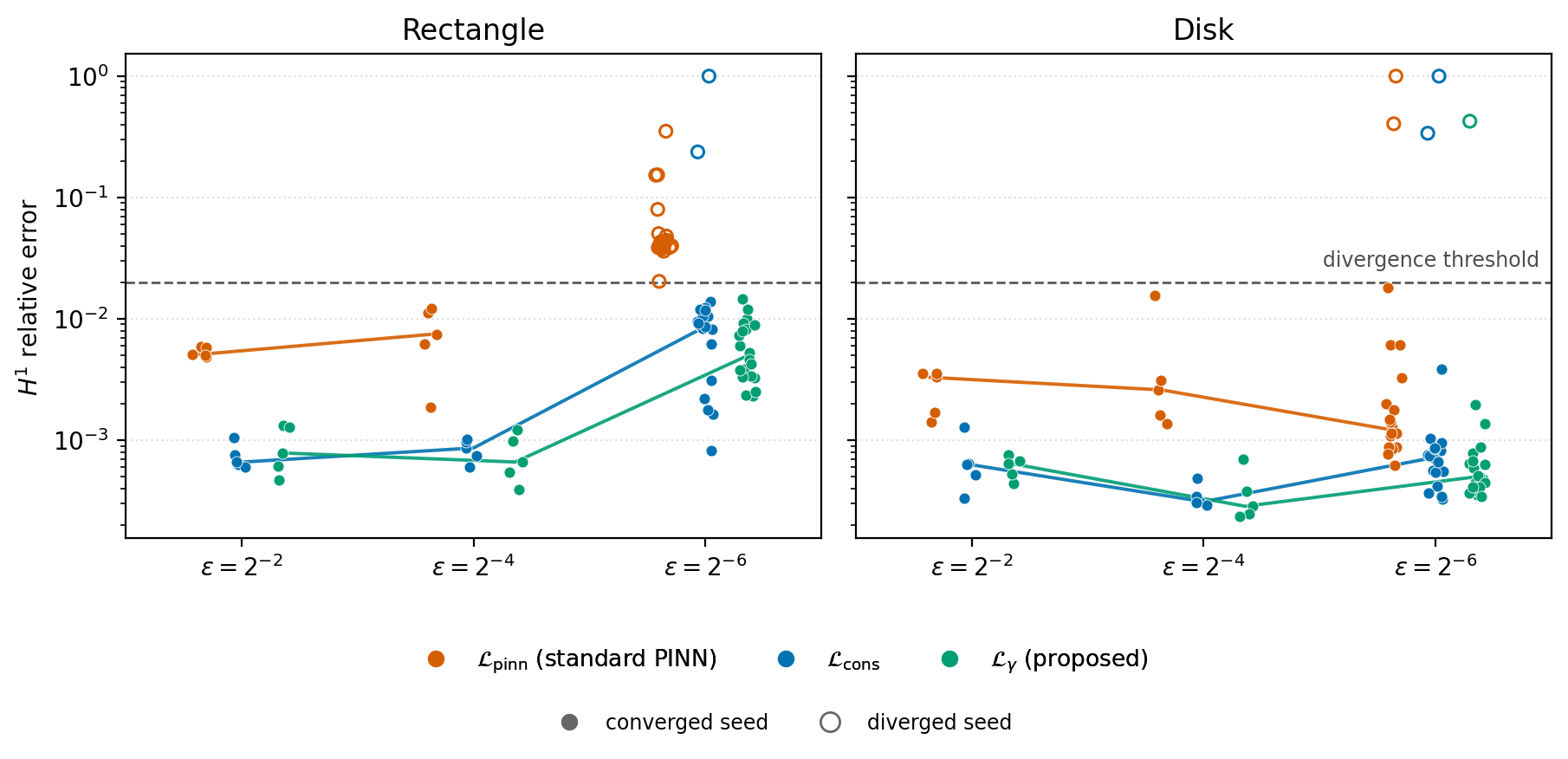}
    \caption{Per-seed $\Hone$ relative error on the rectangle (left) and
    the disk (right) at $\eps = 2^{-s}$, $s \in \{2,4,6\}$. Five seeds at
    $s = 2, 4$; 20 seeds at $s = 6$. Filled markers are converged seeds,
    open markers diverged ($\Hone > 2 \times 10^{-2}$); thin lines pass
    through the median of converged seeds at each $\eps$. On the rectangle
    at $\eps = 2^{-6}$, standard PINN ($\Lpinn$) diverges on every one of
    20 seeds while $\Lgam$ converges on all 20, a complete
    threshold-robust separation. On the disk the same ordering
    ($\Lgam \le \Lcons \le \Lpinn$ in divergence rate) holds but the rates
    are uniformly low and the 95\% confidence intervals overlap. On seeds
    that converge, the three losses are comparable in accuracy across the
    range; the separation is in training reliability, not in the accuracy
    of a converged run.}
    \label{fig:headline}
\end{figure}

\subsection{Experiment 3: Spatial error distribution}
\label{sec:exp-spatial}

To complement the integrated error norms of
\Cref{tab:rectangle,tab:disk},
\Cref{fig:spatial-rect,fig:spatial-disk} show the
pointwise absolute error of $\Lpinn$ and $\Lgam$ at $s = 6$ on
the rectangle and the disk (single representative converged seed). On both
domains the $\Lgam$ error is confined to a thin sliver along the outflow
arc, while the $\Lpinn$ error is delocalised across the bulk, consistent
with the layer-overweighting picture of Remark~\ref{rem:l2-fails}.

\begin{figure}[t]
    \centering
    \includegraphics[width=0.99\linewidth]{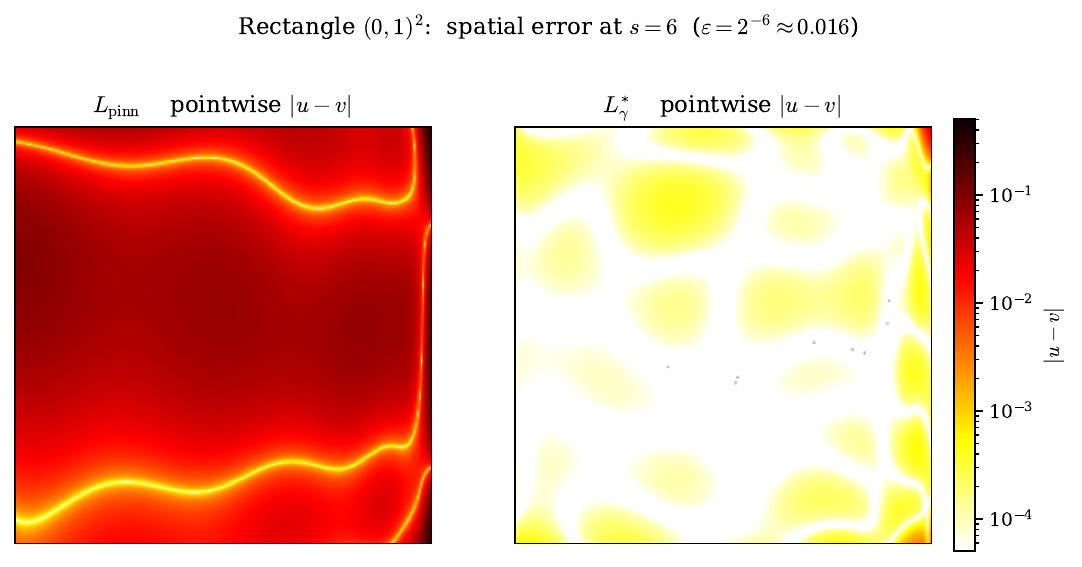}
    \caption{Pointwise absolute error on the rectangle $(0,1)^{2}$ at
     $\eps = 2^{-6}$, single representative seed. Left: standard
    PINN $\Lpinn$. Right: consistent $\Lp{\gamma}$ loss $\Lgam$. Same colour
    scale across both panels. The $\Lpinn$ error is delocalised across the
    bulk, while the $\Lgam$ error is concentrated in the outflow boundary
    layer at the top-right edge.}
    \label{fig:spatial-rect}
\end{figure}

\begin{figure}[t]
    \centering
    \includegraphics[width=0.99\linewidth]{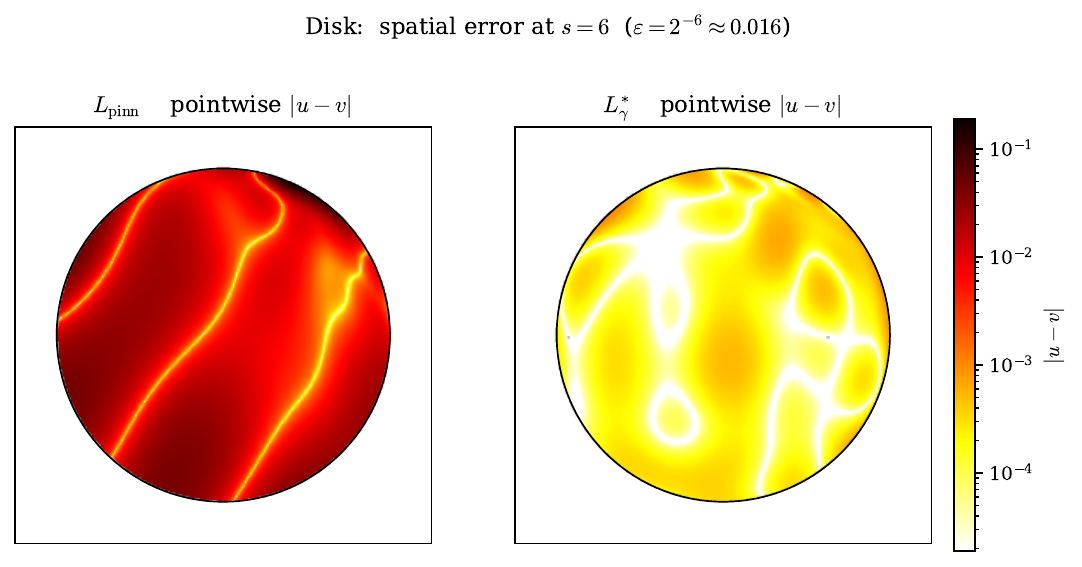}
    \caption{Pointwise absolute error on the disk at $\eps = 2^{-6}$,
    single representative seed. Left: standard PINN
    $\Lpinn$. Right: consistent $\Lp{\gamma}$ loss $\Lgam$. Same colour scale
    across both panels. The $\Lgam$ error is concentrated in a thin arc near
    the outflow region.}
    \label{fig:spatial-disk}
\end{figure}

\section{Concluding Remarks}
\label{sec:conclusion}
This work carries the consistent CutPINN framework of~\cite{singh2026cutpinn}
from the elliptic setting into stationary convection-diffusion on curved
level-set domains, and shows that a single change to the interior norm is
what the new regime demands. Measuring the pointwise residual in an $\Lp{\gamma}$ norm with
$\gamma = 1 + 1/\log\mtil$, in place of $\Lp{2}$, removes the
layer-overweighting that destabilises standard training once convection
dominates, while the boundary trace machinery transfers unchanged. The
method stays a pure PINN throughout: no finite-element reference solution,
no adaptive mesh, and no problem-specific stabilisation. Across the
rectangle and the disk down to $\eps = 2^{-6}$, the $\Lp{\gamma}$ interior
trains reliably on every seed, while the $\Lp{2}$ interiors grow seed-fragile
and diverge on an increasing share of initialisations as the layer sharpens;
where they do converge, all three losses deliver comparable accuracy, so the
gain the new norm buys is reliability. An a priori $\Hone$ bound
(Theorem~\ref{thm:apriori-lgamma}), built on the discrete boundary norm equivalence
as a black box, accompanies and explains these observations. Together these
results identify the interior norm as the lever that governs whether
convection-dominated PINN training succeeds, and give a single formulation
that operates across flat and curved geometry alike.

\subsection*{Ongoing and Future Work}
Several directions extend naturally from here. The a priori
bound~\eqref{eq:apriori-lgamma} carries the convection-diffusion stability
constant $\eps^{-1}$, and sharpening it to an $\eps$-uniform estimate in
the singular-perturbation sense is within reach through a
streamline-diffusion stability analysis on the cut domain, together with a
finer treatment of the boundary chord-arc constants near the outflow arc.
It would also be interesting to study the singularly perturbed regime, where
the layer sharpens further and adaptive collocation, complementary to the
interior-norm correction studied here, becomes relevant. 
The $\Lp{\gamma}$ interior already matches or surpasses several established
enhancements in our experiments, including SUPG-style residual
reweighting~\cite{brooks1982streamline}, anisotropic outflow collocation,
and Fourier-feature embeddings; mapping out the problem classes on which each
of these could reinforce, rather than compete with, the consistent loss is a
worthwhile direction in its own right.

The framework also points beyond the stationary problem. The spatial
discretisation developed here transfers directly to the parabolic
convection-diffusion equation
\begin{equation*}
    u_{t} - \eps \Delta u + \bb \cdot \nabla u + c\,u = f
    \quad \text{in } \dom \times (0, T],
\end{equation*}
under a backward-Euler discretisation in time, in which each step is a
stationary problem of the form~\eqref{eq:cdr} with reaction coefficient
$c + (\Delta t)^{-1}$, precisely the setting analysed here. This
time-dependent extension is the subject of ongoing work.

\bibliographystyle{abbrvnat}
\bibliography{ctpinn}

\end{document}